\documentclass[12pt]{article}
\usepackage{amsmath}
\newtheorem{example}{Example}
\newtheorem{theorem}{Theorem}
\newtheorem{proposition}{Proposition}
\newtheorem{lemma}{Lemma}
\newtheorem{remark}{Remark}
\newtheorem{proof}{Proof}
\usepackage{graphicx}
\usepackage{natbib} %comment out if you do not have the package
\usepackage{url} % not crucial - just used below for the URL 
\usepackage{amsfonts}
\usepackage{times}
\usepackage{bm}
\usepackage{bbm}
\usepackage{natbib}
\usepackage{graphicx}
\usepackage{subfigure}
\usepackage{color}

\newcommand{\blind}{0}

\begin{document}

\def\spacingset#1{\renewcommand{\baselinestretch}%
{#1}\small\normalsize} \spacingset{1}

%%%%%%%%%%%%%%%%%%%%%%%%%%%%%%%%%%%%%%%%%%%%%%%%%%%%%%%%%%%%%%%%%%%%%%%%%%%%%%

\if0\blind
{
  \title{\bf Sliced Latin hypercube designs with arbitrary run sizes}
   \author{Jin Xu\\
     College of Liberal Arts and Sciences, National University of Defense Technology\\
 and \\
    Xu He\\
    Academy of Mathematics and System Sciences, Chinese Academy Sciences\\
        and \\
        Xiaojun Duan \\
       College of Liberal Arts and Sciences, National University of Defense Technology\\
       and \\
       Zhengming Wang\\
      College of Advanced Interdisciplinary Research, National University of \\Defense Technology
}
%  \author{Jin Xu\\
%    College of Science, National University of Defense Technology\\
%    and \\
%    Xu He \\
%    Academy of Mathematics and System Sciences, Chinese Academy Sciences\\
%        and \\
%        Xiaojun Duan \\
%        College of Liberal Arts and Sciences, National University of Defense Technology\\
%            and \\
%            Zhengming Wang\\
%            College of Advanced Interdisciplinary Research, National University of Defense Technology \\}
  \maketitle
} \fi

\if1\blind
{
  \bigskip
  \bigskip
  \bigskip
  \begin{center}
    {\LARGE\bf Title}
\end{center}
  \medskip
} \fi

\bigskip
\begin{abstract}
Latin hypercube designs achieve optimal univariate stratifications and are useful for computer experiments. 
Sliced Latin hypercube designs are Latin hypercube designs that can be partitioned into smaller Latin hypercube designs. 
In this work, we give, to the best of our knowledge, the first construction of sliced Latin hypercube designs that allow arbitrarily chosen run sizes for the slices. 
We also provide an algorithm to reduce correlations of our proposed designs. 
\end{abstract}

\noindent%
\vfill

\newpage
\spacingset{2} % DON'T change the spacing!
\section{Introduction}
\label{sec:intro}
Latin hypercube designs are useful for numerical integration and emulation of computer experiments. 
An $n\times p$ matrix is called a Latin hypercube design if each of its columns contains exactly one point in each of the $n$ bins of $(0,1/n],(1/n,2/n],\cdots,((n-1)/n,1]$, which is called the design achieves univariate stratifications.
\cite{mcKay1979a_compar} proposed a method to generate Latin hypercube designs.
\cite{stein1987large} gave that the variance of the sample mean based on Latin hypercube designs can achieve more reduction than independent and identically distributed sampled.
\cite{owen1992a_central} extended Stein's work and proved a central limit theorem. 
\cite{loh1996on} provided some results about the multivariate central limit theorem and the convergence rate for the sample mean based on Latin hypercube designs.
%\citep{santner2013the} 

Sliced Latin hypercube designs~\citep{qian2012sliced_Latin} are Latin hypercube designs that can be partitioned into several smaller Latin hypercube designs. 
Such designs are appealing when computer simulations are carried out in batches, in multi-fidelity, or with both quantitative and qualitative variables. 
In general, running some complex codes under different parameters in different computer is a time-saving method, which is called as computer experiment in batches.
%The computer experiments in batches are using the different designs to run the same code in some computers. 
Each slice of the sliced Latin hypercube designs can be used to each batch, which makes both the design on each computer and the whole design can achieve optimal one-dimensional uniformity.
The experiments with both quantitative and qualitative variables are very common.
\cite{deng2015design} proposed a new type of designs, marginally coupled designs, for this problem. 
Sliced Latin hypercube designs are desirable to deal with this problem.
Namely, we arrange each slice of designs to each combination of qualitative variables.

While most existing methods generate designs with equal batch sizes, 
in many applications sliced designs with unequal run sizes are needed. 
For instance, when simulations are carried out from multiple computers, it is desirable to assign more runs to faster computers;
to integrate a computer model with one qualitative factor that is not uniformly distributed, it is most efficient to assign more runs to levels with higher probability; 
to emulate computer experiments with tunable accuracy, it was suggested in \citet{he2017optimization} to use more low-accuracy runs than high-accuracy runs. 

In this paper, we give, to the best of our knowledge, the first construction of sliced Latin hypercube designs that allow the run sizes to be chosen arbitrarily. 
Before this work, \citet{Yuan:2017} and \citet{xu2018sliced} constructed sliced Latin hypercube designs with certain types of unequal run sizes. 
Flexible sliced designs~\citep{kong2017flexible} allow flexibly chosen run sizes but are not Latin hypercube designs. 

It is commonly believed that Latin hypercube designs with uncorrelated or nearly uncorrelated columns are more advantageous than average Latin hypercube designs~\citep{owen1994controlling}. 
Inspired by the method of reducing correlations of equal-size sliced Latin hypercube designs~\citep{chen2018controlling}, 
we also provide an algorithm to reduce correlations of our proposed designs. 
Numerical results suggest that this leads to improved performance in some circumstances. 

The remainder of the article is organized as follows.
The constructions for sliced Latin hypercube designs with arbitrary run sizes are given in Section 2.
Section 3 provides an algorithm to reduce the column correlation of designs.
Section 4 gives some numerical illustrations.
Section 5 concludes this paper.
All proofs are deferred to the appendix.

\section{Construction}

\label{sec:con} 

For $a \in R$, let $\lceil a\rceil$ denote the smallest integer no less than $a$. 
We propose generating the sliced Latin hypercube design in $p$ dimensions with $t$ slices of sizes $n_1,\cdots,n_t$ by the following three steps. 

\begin{itemize}
\item[Step 1:] Initialize $S_0=G_1=\cdots=G_t=\emptyset$.
\item[Step 2:] For $i$ from 1 to $n= \sum_{i=1}^t n_i$, let $S_{i,0} = S_{i-1} \cup \{i\}$ and compute 
\[ \delta_i = \sum_{j=1}^{t} \left\{ \lceil n_j(i+1/2)/n\rceil - \lceil n_j(i-1/2)/n\rceil \right\}. \] 
If $\delta_i > 0$, for $j$ from 1 to $\delta_i$, 
let $k$ be the $j$th smallest integer of set $\{z:\lceil n_z(i+1/2)/n\rceil - \lceil n_z(i-1/2)/n\rceil = 1\}$  
and $u$ be the smallest integer in $S_{i,j-1}$ such that $\lceil n_k(u-1/2)/n\rceil = \lceil n_k(i-1/2)/n\rceil$, 
add $u$ to $G_k$, and let $S_{i,j} = S_{i,j-1} \setminus \{u\}$. 
Let $S_i = S_{i,\delta_i}$ and continue to the next $i$. 
\item[Step 3:] For $j$ from 1 to $t$, uniformly permute $G_j$ for $p$ times and obtain $h_{j,1},\cdots,h_{j,p}$ such that all permutations are carried out independently. 
For $l$ from 1 to $p$, stack $h_{1,l},\cdots,h_{t,l}$ together, divide them by $n$, and subtract them by $1/(2n)$ to obtain the $l$th column of the design. 
\end{itemize}

To better understand the algorithm, 
we now present a simple example. 

\begin{example}\label{exa:1}
Consider $t = 3,$ $n_1=2$, $n_2=5$, $n_3=10$, $n=17$, and $p=3$. 
Here, $(\delta_1,\cdots,\delta_n) = (0,1,2,0,1,0,2,0,2,2,0,1,0,1,1,0,3)$. 
Since $\delta_1=0$, we have $S_1 = \{1\}$. 
For $i=2$, $S_{2,0} = S_1 \cup \{2\} = \{1,2\}$, $\delta_i=1$, 
$k=3$ is the only integer satisfying $\lceil n_k(i+1/2)/n\rceil - \lceil n_k(i-1/2)/n\rceil = 1$, 
and $u=1$ is the smallest integer among the two integers satisfying $\lceil n_3(u-1/2)/n\rceil = \lceil n_3(i-1/2)/n\rceil$. 
Thus, $S_2 = S_{2,1} = S_{2,0} \setminus \{1\} = \{2\}$ and we assign 1 to $G_3$. 
For $i=3$, $S_{3,0} = \{2,3\}$, $\delta_i=2$, 
and both $k=2$ and $k=3$ make $\lceil n_k(i+1/2)/n\rceil - \lceil n_k(i-1/2)/n\rceil = 1$. 
We first set $k=2$ and find that $u=2$ is the smallest integer satisfying $\lceil n_2(u-1/2)/n\rceil = \lceil n_2(i-1/2)/n\rceil$. 
Thus, $S_{3,1} = \{3\}$ and we assign 2 to $G_2$. 
We then set $k=3$. Luckily, the only number in $S_{3,1}$, $u=3$, makes $\lceil n_3(u-1/2)/n\rceil = \lceil n_3(i-1/2)/n\rceil$. 
Thus, $S_3 = S_{3,2} = \emptyset$ and we assign 3 to $G_3$. 
After going through all $i$, we finally obtain $S_n=\emptyset$, $G_1=\{7,14\}$, $G_2=\{2,5,9,12,16\}$, and $G_3=\{1,3,4,6,8,10,11,13,15,17\}$.
Randomly permuting $G_1$, $G_2$ and $G_3$, we obtain $h_{1,1} = (7,14)$, $h_{2,1} = (12,2,16,9,5)$, and $h_{3,1} = (15, 6, 17, 11, 1, 13, 10, 3, 4, 8)$. 
Thus, the first column of the final design is $(13, 27, 23, 3, 31, 17, 9, 29, 11, 33, 21, 1, 25, 19, 5, 7, 15)^{T}/34$. 
Similarly, we can obtain other columns of the design. 
\end{example}

\begin{remark}
The algorithm is valid only if in Step~2 there is at least one element in $S_{i,j-1}$ such that $\lceil n_k(u-1/2)/n\rceil = \lceil n_k(i-1/2)/n\rceil$. 
Proposition~\ref{pro:set-non-empty:mid} below ensures this. 
\end{remark}

\begin{proposition}
\label{pro:set-non-empty:mid}
For any $i = 1,\cdots,n$, $\delta_i >0$, and $j = 1,\cdots,\delta_i$, 
there is at least one element of $S_{i,j-1}$ that makes $\lceil n_k(u-1/2)/n\rceil = \lceil n_k(i-1/2)/n\rceil$. 
\end{proposition}

All of the proofs are given in the Appendix. 
Theorem~\ref{the:slhd} below shows the generated designs are sliced Latin hypercube designs. 

\begin{theorem}
\label{the:slhd}
Let $H$ denote an arbitrary column of a design generated from the proposed algorithm. 
Then, (i) $H$ is a permutation of $\{1/(2n),3/(2n),\cdots,(2n-1)/(2n)\}$; 
and (ii) for $i = 1,\cdots,t$, the $(\sum_{k=1}^{i-1} n_k +1)$th to the $(\sum_{k=1}^i n_k)$th entry of $H$ have exactly one element in each of the $n_i$ bins of $(0,1/n_i],\cdots,((n_i-1)/n_i,1]$.
\end{theorem}

\begin{remark}
In contrast to ``randomized'' Latin hypercube designs with entries that take arbitrary values in $[0,1]$, our algorithm yields ``midpoint'' Latin hypercube designs with entries that locate at the center of the bins of $(0,1/n],\cdots,((n-1)/n,1]$. 
One can view our algorithm as assigning elements of the one-dimensional midpoint Latin hypercube design $\{1/(2n),\cdots,(2n-1)/(2n)\}$ 
to $G_1,\cdots,G_t$, such that each of the $n_i$ bins of $(0,1/n_i],\cdots,((n_i-1)/n_i,1]$ contains exactly one element of $G_i$ for $i=1,\cdots,t$. 
We focus on midpoint designs because, unlike the case with equal run sizes, not every one-dimensional Latin hypercube design can be partitioned at will. 
For instance, consider the case with $t=3$, $n_1=1$, $n_2=n_3=3$, $n=7$, and $H=(0.1,0.2,0.3,0.5,0.7,0.8,0.9)^T$. 
It is not difficult to verify that each of the seven bins of $(0,1/7],\cdots,(6/7,1]$ contains exactly one point of $H$, but there is no partition of $H$ to $G_1$, $G_2$, and $G_3$ that fulfills the property of Theorem~\ref{the:slhd}(ii). 
Interestingly, when $H = \{1/(2n), \cdots, (2n-1)/(2n)\}$, at least one valid assignment always exists, 
and Theorem~\ref{the:slhd} is the first result indicating this.  
Furthermore, from numerical results shown in Section~\ref{sec:sim}, midpoint Latin hypercube designs are usually as good as or even better than randomized Latin hypercube designs. 
\end{remark}
\section{Reducing correlations}
\label{sec:corr}
\cite{chen2018controlling} gives a method to control column-wise correlations of sliced Latin hypercube designs. 
In this section, we provide an algorithm to reduce the correlations between each column of the designs proposed in Section 2.

Let $D_{j,k}$ denote the $j$th slice of the $k$th column of $D$, a sliced design obtained from our algorithm in Section~\ref{sec:con}. 
We can further reduce the correlations of $D$ using the following five steps. 

\begin{itemize}
\item[Step 1:] For $j$ from $1$ to $t$, $k$ from $2$ to $p$, and $l$ from $1$ to $k-1$, 
fit a simple linear regression model with $D_{j,l}$ being the response and $D_{j,k}$ being the only covariate besides the intercept, 
and replace $D_{j,l}$ with the residual. 
\item[Step 2:] For $j$ from $1$ to $t$, $k$ from $1$ to $p$, and $u$ from $1$ to $n_j$, use the $u$th smallest element of $G_j$, subtracted by $1/2$ and divided by $n$, to replace the $u$th smallest element of $D_{j,k}$. 
\item[Step 3:] For $j$ from $1$ to $t$, $k$ from $p-1$ to $1$, and $l$ from $p$ to $k+1$, 
fit a simple linear regression model with $D_{j,l}$ being the response and $D_{j,k}$ being the only covariate besides the intercept, 
and replace $D_{j,l}$ with the residual. 
\item[Step 4:] For $j$ from $1$ to $t$, $k$ from $1$ to $p$, and $u$ from $1$ to $n_j$, use the $u$th smallest element of $G_j$, subtracted by $1/2$ and divided by $n$, to replace the $u$th smallest element of $D_{j,k}$. 
\item[Step 5:] Iterate Steps~1-4 nine more times. 
\end{itemize}
Here,  replacing $D_{j,l}$ with the residual means
$$D_{j,l}=D_{j,l}-(D_{j,k}-\bar{D}_{j,k})\rho(D_{j,k},D_{j,l})\sigma(D_{j,l})/\sigma(D_{j,k}),$$
where $\rho(D_{i,j},D_{i,k})$ is the sample correlation of $D_{i,j}$ and $D_{i,k}$, $\sigma(D_{i,k})$ and $\sigma(D_{i,j})$ are the standard deviations of the two vectors and $D_{i,j}-\bar{D}_{i,j}$ amounts to $D_{i,j}$
minus its mean times a vector of 1s. 

\begin{remark}
\cite{chen2018controlling} controls  column-wise correlations of each slices, then combine them to obtain a sliced Latin hypercube designs.
Similarly, we reduce the correlations of each slices, and combine them to let the each slice and the whole design be Latin hypercube designs.
\end{remark}

\begin{remark}
The algorithm is said to converge if the root mean square correlation among columns, which is defined in \cite{owen1994controlling}, stops decreasing.
The root mean square correlation is 
\begin{equation*}
\rho_{\rm rms} (D)= \left(\frac{\sum_{1\leq j<k\leq p}(\rho(D_{:,j},D_{:,k}))^2}{p(p-1)/2}\right)^{1/2}
\end{equation*}
where $D$ is a design in $p$ factors, $D_{:,j}$ and $D_{:,k}$ are the $j$th and $k$ columns of $D$, respectively.
We stop the above algorithm after 10 iterations because from our experience it already warrants convergence. 
\end{remark}

We give an example to illustrate this algorithm.
\begin{example}
Consider $t=2,n_1=6,n_2=7,p=3$, using the algorithm 1 in section 2 to generate the initial design $D:$
\begin{equation}
\label{equ:initial-D}
D=1/26\times\left[
\begin{array}{cccccc|ccccccc}
19 & 23 & 11 & 5 & 15 & 1  & 25 & 9  & 7  & 3  & 17 & 13 & 21 \\
15 & 23 & 11 & 5 & 1  & 19 & 9  & 13 & 21 & 17 & 3  & 7  & 25 \\
11 & 15 & 19 & 5 & 23 & 1  & 17 & 21 & 9  & 25 & 7  & 13 & 3 
\end{array}
\right]^T
\end{equation}
In the Step 1, when $j=1,k=2,l=1$, we have 
$D_{1,1}=(19 ,23,11,5,15,1 )^T/26$ and $D_{1,2}=(15, 23, 11, 5, 1, 19)^T/26$.
Clearly,
$\sigma(D_{1,1})=\sigma(D_{1,2})=0.7189,\rho(D_{1,2},D_{1,1})=0.2328.$
Then, renew $D_{1,1}$ to get that
\begin{equation*}
\begin{split}
D_{1,1}
&=D_{1,1}-(D_{1,2}-\bar{D}_{1,2})\rho(D_{1,2},D_{1,1})\sigma(D_{1,1})/\sigma(D_{1,2})\\
&=(0.7068,     0.7891,     0.4350 ,    0.2580  ,   0.6784,   -0.0212)^T.
\end{split}
\end{equation*}
Similarly, when $j=1,k=3,l=1$, we have 
$$D_{1,1}=(0.7632, 0.8196, 0.2606, 0.3710,  0.3170, 0.31464)^T.$$
%$$D_{2,1}=(0.9228, 0.3413,0.3321, 0.1444, 0.5643,  0.4443,  0.9045)^T.$$
Then, Step 1 gives that
\begin{equation*}
D=\left[
\begin{array}{ccc}
0.7633 & 0.5583 & 11/26 \\
0.8196 & 0.9218 & 15/26 \\
0.2606 & 0.5161 & 19/26 \\
0.3710 & 0.0900 & 5/26  \\
0.3170 & 0.1872 & 23/26 \\
0.3146 & 0.5727 & 1/26  \\
\hline
1.0273 & 0.3681 & 17/26 \\
0.4886 & 0.5476 & 21/26 \\
0.1816 & 0.7784 & 9/26  \\
0.3345 & 0.7271 & 25/26 \\
0.5279 & 0.0733 & 7/26  \\
0.4890 & 0.2656 & 13/26 \\
0.6050 & 0.8938 & 3/26 
\end{array}
\right]
\end{equation*}
In Step 2, for $j = 1$, we have $G_1=\{  1/26,  5/26, 11/26,  15/26, 19/26, 23/26\}$,
then $$D_{1,1}=(0.7632, 0.8196, 0.2606, 0.3710,  0.3170, 0.31464)^T$$ is replaced with
$$D_{1,1}=(19/26, 23/26,1/26, 15/26, 11/26, 5/26)^T.$$
Similarly, we have 
\begin{equation}
\label{equ:forward-D}
D=1/26\times\left[
\begin{array}{cccccc|ccccccc}
19 & 23 & 1  & 15 & 11 & 5  & 25 & 9  & 3  & 7  & 17 & 13 & 21 \\
15 & 23 & 11 & 1  & 5  & 19 & 9  & 13 & 21 & 17 & 3  & 7  & 25 \\
11 & 15 & 19 & 5  & 23 & 1  & 17 & 21 & 9  & 25 & 7  & 13 & 3 
\end{array}
\right]^T
\end{equation}
after Step 2.
%For the design $D$ changes from \eqref{equ:initial-D} to \eqref{equ:forward-D}, the root mean square correlation reduces from 0.067 to 0.053.
Finally, we obtain
\begin{equation}
\label{equ:final-D}
D=1/26\times\left[
\begin{array}{cccccc|ccccccc}
19 & 23 & 1  & 15 & 11 & 5  & 25 & 9  & 3  & 7  & 17 & 13 & 21 \\
15 & 23 & 11 & 1  & 5  & 19 & 13 & 9  & 17 & 21 & 3  & 7  & 25 \\
11 & 15 & 19 & 5  & 23 & 1  & 21 & 17 & 7  & 25 & 9  & 13 & 3 
\end{array}
\right]^T
\end{equation}
Figure \ref{fig:reduce-corr} shows that the root mean square correlations of the each slice and the whole design have a distinct reduction.

%For the design $D$ changes from \eqref{equ:initial-D} to \eqref{equ:final-D}, the root mean square correlation reduces from 0.067 to 0.021.

%0.587116135423344
%0.324731831402145
%0.0669303224248279
%
%0.0738593322067479
%0.204016216367870
%0.0531638690979350
%
%0.0738593322067479
%0.0139087618887299
%0.0205591112184519
\begin{figure}
\centering
\includegraphics[width=0.65\linewidth]{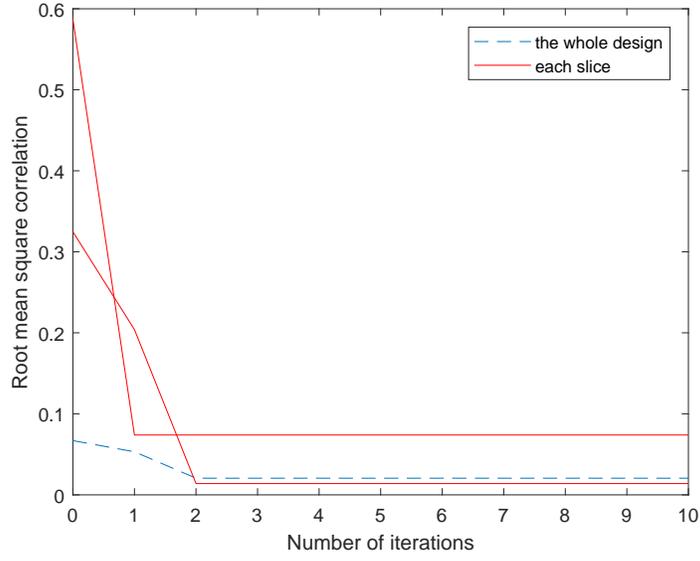}
\caption{The changes of root mean square correlations of the each slice and the whole design}
\label{fig:reduce-corr}
\end{figure}
\end{example}

\section{Numerical comparison}
\label{sec:sim}

We now demonstrate the usefulness of our proposed sliced designs in numerically integrating the two test functions used in \citet{qian2012sliced_Latin},   
\begin{eqnarray*}
f_1(x) & = & \log(x_1x_2x_2x_4x_5), \\
f_2(x) & = & \log\left( x_1^{-1/2} + x_2^{-1/2} \right).  
\end{eqnarray*}

Assume we have $t$ computers to evaluate the functions and from time constraints we can arrange at most $n_1,\cdots,n_t$ runs for the computers separately. 
We have at least three choices to solve the problem, as follows. 
First, we can use a single design with $n=\sum_{i=1}^t n_i$ runs and assign the runs randomly to the $t$ computers. 
We consider using an ordinary Latin hypercube design~\citep{mcKay1979a_compar}, its midpoint modification, and its correlation-controlled extension~\citep{owen1994controlling} for this approach. 
Second, we can combine $t$ independently generated Latin hypercube designs with sizes $n_1,\cdots,n_t$, separately, and assign one design to each computer. 
Third, we can use a flexible sliced design~\citep{kong2017flexible} or our newly proposed sliced Latin hypercube design and assign one slice to each computer. 
These methods are shown as follows.
\begin{itemize}
\item[\textbf{RLH}] single randomized Latin hypercube design with $n$ runs;
\item[\textbf{MLH}] single midpoint Latin hypercube design with $n$ runs;
\item[\textbf{CLH}] single correlation-controlled Latin hypercube design with $n$ runs;
\item[\textbf{IMLH}] $t$ independent midpoint Latin hypercube designs with $n_1,\cdots,n_t$ runs, respectively; 
\item[\textbf{ICLH}]  $t$ independent correlation-controlled Latin hypercube designs with $n_1,\cdots,n_t$ runs, respectively; 
\item[\textbf{FSD}] flexible sliced design in $t$ slices, and its $i$th slice contains $n_i$ runs;
\item[\textbf{SLH}] the proposed sliced Latin hypercube design in $t$ slices, and its $i$th slice contains $n_i$ runs; 
\item[\textbf{CSLH}] the proposed sliced Latin hypercube design with reduced correlations in $t$ slices, and its $i$th slice contains $n_i$ runs.
\end{itemize}
Under all approaches, we estimate the mean function value using the averaged output value among completed computer trials. 

We compare the methods using two scenarios. 
In the first scenario, all of the functional evaluations terminate correctly and we obtain all $n$ output values. 
In the second scenario, one random computer fails and we obtain all other output values. 
For $f_1$, we assume $t=4$, $n_1=17$, $n_2=13$, $n_3=11$, and $n_4=7$; 
for $f_2$, we assume $t=3$, $n_1=9$, $n_2=7$, and $n_3=6$. 
We repeat the procedure 10,000 times and report the averaged root-mean-square estimation error in Table~\ref{tab:result}. 
%\begin{table}
%\caption{Root-mean-square estimation error on mean output value. \label{tab:result}}
%\begin{center}
%\begin{tabular}{cccccccccc}
%Function & Scenario & RLH & MLH & CLH & IMLH & ICLH & FSD & SLH & CSLH \\
%$f_1$ & 1 & $0\cdot0487$&	$0\cdot0360$&	$0\cdot0360$&	$0\cdot1428$&$0\cdot1428$&	$0\cdot0971$&	$0\cdot0360$&	$0\cdot0360$ \\
%      & 2 & $0\cdot1941$&	$0\cdot1851$&	$0\cdot1845$&	$0\cdot1442$&	$0\cdot1442$&	$0\cdot1132$&	$0\cdot0958$&	$0\cdot0958$ \\
%$f_2$ & 1 & $0\cdot0121$&	$0\cdot0060$&	$0\cdot0041$&	$0\cdot0117$&	$0\cdot0110$&	$0\cdot0194$&	$0\cdot0061$&	$0\cdot0042$ \\
%      & 2 & $0\cdot0363$&	$0\cdot0322$&	$0\cdot0319$&	$0\cdot0122$&	$0\cdot0112$&	$0\cdot0239$&	$0\cdot0099$&	$0\cdot0075$ 
%\end{tabular}
%\end{center}
%\end{table}

\begin{table}
\caption{Root-mean-square estimation error on mean output value. \label{tab:result}}
\begin{center}
\begin{tabular}{cccccccccc}
Function & Scenario & RLH & MLH & CLH & IMLH & ICLH & FSD & SLH & CSLH \\
$f_1$ & 1 & $0.0487$&	$0.0360$&	$0.0360$&	$0.1428$&$0.1428$&	$0.0971$&	$0.0360$&	$0.0360$ \\
      & 2 & $0.1941$&	$0.1851$&	$0.1845$&	$0.1442$&	$0.1442$&	$0.1132$&	$0.0958$&	$0.0958$ \\
$f_2$ & 1 & $0.0121$&	$0.0060$&	$0.0041$&	$0.0117$&	$0.0110$&	$0.0194$&	$0.0061$&	$0.0042$ \\
      & 2 & $0.0363$&	$0.0322$&	$0.0319$&	$0.0122$&	$0.0112$&	$0.0239$&	$0.0099$&	$0.0075$ 
\end{tabular}
\end{center}
\end{table}

Observed from the results, midpoint Latin hypercube designs are usually better than ordinary Latin hypercube designs. 
Reducing correlation helps for $f_2$ but has no effect for $f_1$. 
Both with and without reducing correlations, the proposed sliced designs perform the best for all functions and scenarios. 
Single Latin hypercube designs are as good as the proposed designs in the first scenario but much worse in the second scenario. 
Independent Latin hypercube designs and flexible sliced designs are inferior to the proposed designs in both scenarios. 
These observations suggest that the proposed new designs, while allowing flexible run sizes, achieve the same level of variance reduction as ordinary sliced Latin hypercube designs. 

\section{Conclusion}
\label{sec:conc}
We propose, to the best of our knowledge, the first construction of sliced Latin hypercube designs that allow the run sizes to be chosen arbitrarily. 
Moreover, we provide an algorithm to reduce correlations of our proposed designs. 
Numerical results suggest that this leads to improved performance in some circumstances. 
\appendix
%\appendixone
\section*{Appendix 1}
\subsection*{Proofs}

For a set $A$, let $\text{card}(A)$ denote its cardinality. 

%Before proving Proposition \ref{pro:set-non-empty:mid}, we first give a lemma.
\begin{lemma}
\label{lem:number_of_set:mid}
Assume $t,n_1,n_2\cdots,n_t$, $a$ and $l$ are integers with $a+l\leq n$ and $n = \sum_{j=1}^{t}n_j$, and
let 
\[ \Omega=\{(i,j):a+1\leq i\leq a+b,1\leq j\leq t,\lceil n_j(a-1/2)/n\rceil<\lceil n_j (i-1/2)/n\rceil<\lceil n_j(i+1/2)/n\rceil\}.\] 
Then, $\text{card}(\Omega)\leq l$. 
\end{lemma}

\begin{proof} %[Proof of Lemma \ref{lem:number_of_set:mid}]
Let
$$\Omega_j = \{i:a+1\leq i\leq a+l, \lceil n_j(a-1/2)/n\rceil<\lceil n_j (i-1/2)/n\rceil<\lceil n_j(i+1/2)/n\rceil\},$$ 
and then $\Omega = \cup_{j= 1}^t(\Omega_j,j)$ and $\text{card}(\Omega) = \sum_{j= 1}^t \text{card}(\Omega_j)$.

When $\lceil n_j(a+l+1/2)/n\rceil = \lceil n_j(a-1/2)/n\rceil$, we have $\text{card}(\Omega_j) = 0 $. 
Because 
%\begin{equation*}
$\lceil n_j(a+l+1/2)/n\rceil \geq  \lceil n_j(a-1/2)/n\rceil+\lceil n_j(l+1)/n\rceil-1$,
%\end{equation*}
we have $\lceil n_j(l+1)/n\rceil=1$ and therefore 
\begin{equation*}
\label{equ:lem:num:mid:1}
\text{card}(\Omega_j)=\lceil n_j(l+1)/n\rceil -1.
\end{equation*}
When $\lceil n_j(a+l+1/2)/n\rceil > \lceil n_j(a-1/2)/n\rceil$, we have $\text{card}(\Omega_j) = \lceil n_j(a+l+1/2)/n\rceil - \lceil n_j(a-1/2)/n\rceil-1.$
Because $\lceil n_j(a+l+1/2)/n\rceil \leq \lceil n_j(l+1)/n\rceil +\lceil n_j(a-1/2)/n\rceil$, 
we have 
\begin{equation*}
\text{card}(\Omega_j)\leq  \lceil n_j(l+1)/n\rceil-1.
\end{equation*}
Combining the two cases, we have
\begin{equation*}
\text{card}(\Omega) = \sum_{j= 1}^t \text{card}(\Omega_j)\leq \sum_{j= 1}^t\{\lceil n_j(l+1)/n\rceil-1\}<\sum_{j= 1}^t\{n_j(l+1)/n\} = l+1.
\end{equation*}
\end{proof}
%Next, we give the proof of Proposition \ref{pro:set-non-empty:mid} as follows.
\begin{proof} [of Proposition \ref{pro:set-non-empty:mid}]
For any $i$ with $\delta_i>0$, let $\rho_i=\{l:\lceil n_l(i+1/2)/n\rceil-\lceil n_l(i-1/2)/n\rceil=1\}$ and $\rho_{i,j}$ be the $j$th smallest element of $\rho_i$.
For any $l\in\rho_i$, sort the $\delta_{i}$ sets $\{u:\lceil n_l(u-1/2)/n\rceil = \lceil n_l(i-1/2)/n\rceil\}$ by its cardinality in descending order; 
for sets with the same cardinality, sort them in ascending order of $l$. 
Let $\Gamma_{i,j}$ and $\pi_{i,j}$ denote the $j$th set and the corresponding $l$, respectively. 
%Note that $\Gamma_{i,j} = \emptyset$ if and only if $j>\delta_i.$

%Clearly, $\text{card}(\Gamma_{i,1}) \geq \cdots \geq \text{card}(Gamma_{i,1})$ and 

%Let $\Gamma_{i,j}$ denote the set $\{u:\lceil n_{\pi_{i,j}}(u-1/2)/n\rceil = \lceil n_{\pi_{i,j}}(i-1/2)/n\rceil\}$, 
%where $\pi_{i,j}$ is the $j$th entry of vector $\pi_i$.
%The $\delta_i$ dimensional vector $\pi_i$ is defined by arranging set $\{l: \lceil n_l(i+1/2)/n\rceil-\lceil n_l(i-1/2)/n\rceil=1\}$, which satisfies that $\text{card}(\Gamma_{i,j})>\text{card}(\Gamma_{i,j+1})$ or $\text{card}(\Gamma_{i,j})=\text{card}(\Gamma_{i,j+1})$ but $\pi_{i,j}<\pi_{i,j+1}$. 
In the follows we show the set $\Gamma_{i,j} \cap S_{i,0}$ contains at least $\delta_i-j+1$ elements for any $(i,j)$ with $\delta_i>0$ and $j \leq \delta_i$.
%In the rest of the proof we prove it.
%As defined in Algorithm \ref{alg:existence:mid}, the set $S_r^*$ denotes $S$ where we have put $r$ in and have not deleted $c_{r,i_j}$ from and the set $S_{r,\tau_r(j)}$ denotes $S_r^*\setminus \{c_{r,\tau_r{1}},\cdots,c_{r,\tau_r{j-1}}\}$.
For any $1\leq i\leq n$, 
\begin{equation*}
\begin{split}
i- \sum_{u=1}^{i}\delta_u 
& = i-  \sum_{l=1}^{t}\sum_{u=1}^{i}\{\lceil n_l(u+1/2)/n\rceil - \lceil n_l(u-1/2)/n\rceil\}\\
& = i-  \sum_{l=1}^{t}\{\lceil n_l(i+1/2)/n\rceil - \lceil n_l/(2n)\rceil\}\\
& = t-1/2-\sum_{l=1}^{t}\xi_l,
\end{split}
\end{equation*}
where $\xi_l = \lceil n_l(i+1/2)/n\rceil-n_l(i+1/2)/n\in[0,1)$. 
Furthermore, $\sum_{l=1}^{t}\xi_l+1/2 < t+1/2$ is an integer. 
Therefore, $\sum_{l=1}^{t}\xi_l+1/2\leq t$ and thus $i- \sum_{u=1}^{i}\delta_u \geq 0$.
Therefore, $\text{card}(S_{i,0}) = i - \sum_{u=1}^{i-1}\delta_u \geq \delta_i$.
%Depending on the definition of the set $\pi_{i_j,\lceil n_{i_k}r/n\rceil}$ and $\delta_{r,i_j} = 1$, we have $\max\{\pi_{i_j,\lceil n_{i_j}r/n\rceil}\} = r$ and all the integers between $\max\{\pi_{i_j,\lceil n_{i_j}r/n\rceil}\}$ and $\min\{\pi_{i_j,\lceil n_{i_j}r/n\rceil}\}$ are in $\pi_{i_j,\lceil n_{i_k}r/n\rceil}$.

Consider arbitrary $i$ and $j$ with $\delta_i>0$ and $1\leq j\leq \delta_i$. 
Suppose $\text{card}(\Gamma_{i,j} \cap S_{i,0} ) = m < \delta_i-j+1$. 
Because $\text{card}(S_{i,0})\geq \delta_i\geq \delta_i-j+1 > m$, $S_{i,0}$ has at least $m+1$ elements.
Therefore, the $m$th largest element of $S_{i,0}$ is in $\Gamma_{i,j}$ while the $(m+1)$th largest element is not in $\Gamma_{i,j}$.
Let $q $ denote the $(m+1)$th largest element of $S_{i,0}$, we have $q\notin\Gamma_{i,j}$ and $i$ is the largest element of the set $S_{i,0}$.
Clearly, $\{q+1,\cdots,i-1\}$ contains $m-1$ elements of $S_{i,0}$. 
Namely, the set $\{u: q+1\leq u\leq i-1 , u\notin S_{i,0}\}$ has $i-q-m$ elements. 
For any $v\in\{u: q+1\leq u\leq i-1, u\notin S_{i,0}\}$, 
there exists a pair of $(w,l)$ which satisfies $q+1\leq w\leq i-1$, $l\in\rho_w$, and $v$ is the smallest number of $\Gamma_{w,z_1}\cap S_{w,z_2-1}$, where $z_1$ and $z_2$  make $\pi_{w,z_1}=\rho_{w,z_2}=l$.
Because $q\in S_{i,0}$, $S_{i,0}\subset S_{w,z_2-1}\cup\{w+1,\cdots,r\}$ and $q <v\leq w$, we have $q\in S_{w,z_2-1}$.
Because $v$ is the smallest number of $\Gamma_{w,z_1}\cap S_{w,z_2-1}$ and $q<v$, we have $q\notin \Gamma_{w,z_1}\cap S_{w,z_2-1}$.
Namely, $q\in S_{w,z_2-1}$ while $q\notin \Gamma_{w,z_1}\cap S_{w,z_2-1}$. 
Therefore, $\lceil n_{l}(q-1/2)/n\rceil < \lceil n_{l}(w-1/2)/n\rceil.$
Consequently, for each element of set $\{u: q+1\leq u\leq i-1, u\notin S_{i,0}\}$, there exists a pair of $(w,l)$ which satisfies $q+1\leq w\leq i-1$ and $\lceil n_{l}(q-1/2)/n\rceil < \lceil n_{l}(w-1/2)/n\rceil< \lceil n_{l}(w+1/2)/n\rceil$. 
Clearly, the $i-q-m$ pairs of $(w,l)$ are different from each other.
Therefore, the set 
$$\{(u,l): q+1\leq u\leq i-1, \lceil n_l(q-1/2)/n\rceil<\lceil n_l (u-1/2)/n\rceil<\lceil n_l(u+1/2)/n\rceil,l = 1,\cdots,t\}$$
has at least $i-q-m$ elements.

Furthermore, $\Gamma_{i,j}\supseteq\Gamma_{i,j+1}$ for any $j = 1,\cdots,\delta_i-1$. 
Because $q\notin\Gamma_{i,j}$, we have $q\notin \Gamma_{i,k}$ for $k=j+1,\cdots,\delta_i$.
Namely, $\lceil n_{\pi_{i,k}}(q-1/2)/n\rceil<\lceil n_{\pi_{i,k}}(i-1/2)/n\rceil$ for $k=j+1,\cdots,\delta_i$. 
Therefore, the set $\{(i,l): \lceil n_l(q-1/2)/n\rceil<\lceil n_l(i-1/2)/n\rceil<\lceil n_l(i+1/2)/n\rceil,l = 1,\cdots,t\}$
has at least $\delta_{i}-j+1$ elements.
Therefore, 
\begin{equation*}
\begin{split}
& \text{card}[ \{(u,l): q+1\leq u\leq i, \lceil n_l(q-1/2)/n\rceil<\lceil n_l(u-1/2)/n\rceil<\lceil n_l(u+1/2)/n\rceil,l = 1,\cdots,t\}]\\
& = \text{card}[ \{(i,l): \lceil n_l(q-1/2)/n\rceil<\lceil n_l(i-1/2)/n\rceil<\lceil n_l(i+1/2)/n\rceil\rceil,l = 1,\cdots,t\}\cup\{(u,l): \\
& \quad\quad q+1\leq u\leq i-1, \lceil n_l(q-1/2)/n\rceil<\lceil n_l(u-1/2)/n\rceil<\lceil n_l(u+1/2)/n\rceil,l= 1,\cdots,t\} ]\\
& \geq i-q-m+\delta_i-j+1\\
& > i-q,
\end{split}
\end{equation*}
which is contradictory to Lemma \ref{lem:number_of_set:mid}. 
Therefore, our assumption that $\text{card}(\Gamma_{i,j}\cap S_{i,0}) = m < \delta_i-j+1$ is false. 
Namely, for any $i = 1,\cdots,n$, $\delta_i >0$, and $j = 1,\cdots,\delta_i$, 
we have $\text{card}(\Gamma_{i,j}\cap S_{i,0}) \geq \delta_i-j+1$ 
and therefore there is at least one element of $S_{i,j-1}$ that makes $\lceil n_k(u-1/2)/n\rceil = \lceil n_k(i-1/2)/n\rceil$.
\end{proof}

%We give the proof of Theorem \ref{the:slhd} as follows.
\begin{proof} [of Theorem \ref{the:slhd}]
(i) Because $\cup_{i=1}^n S_{i,0} = \{1,\cdots,n\}$, we have $\cup_{j=1}^t G_j \subset \{1,\cdots,n\} $.
However, $\text{card}(G_j) = \lceil n_j(n+1/2)/n\rceil-\lceil n_j/2/n\rceil = n_j$ and $\sum_{j=1}^t \text{card}(G_j) = n$. 
Therefore, $\cup_{j=1}^t G_j = \{1,\cdots,n\} $ and $G_1,\cdots,G_t$ are disjoint. 
Namely, $G_1,\cdots,G_t$ are a partition of $\{1,\cdots,n\} $ 
and thus each dimension of $D$ is a permutation of $\{1/(2n),3/(2n),\cdots,(2n-1)/(2n)\}$.

(ii) For any $i,j$ with $\lceil n_{j}(i-1/2)/n\rceil < \lceil n_{j}(i+1/2)/n\rceil$, there exists an integer $c \in G_j$ such that $\lceil n_{j}(c-1/2)/n\rceil = \lceil n_{j}(i-1/2)/n\rceil$.
Therefore, $\{\lceil n_j(c-1/2)/n\rceil : c\in G_j\} = \{1,\cdots, n_j\}$ for any $j$.
Consequently, $(h_{j,l}-1/2)/n$ has exactly one element in each of the $n_j$ bins of $(0,1/n_j],\cdots,((n_j-1)/n_j,1]$ for any $j$ and $l$. 
\end{proof}

\bibliographystyle{Chicago}

\bibliography{refs}
\end{document}